\documentclass[12pt]{amsart}
\usepackage{amssymb}
\usepackage{amsmath}
\usepackage{bbm}
\usepackage[dvips]{graphicx}
\usepackage{hyperref}
\usepackage[capitalize,nameinlink,noabbrev,nosort]{cleveref}
\usepackage[dvipsnames]{xcolor}
\usepackage{todonotes}
\usepackage{mathtools}
\usepackage{subcaption}
\usepackage{enumitem}
\usepackage{fullpage}
\usepackage{ytableau}
\usepackage{tikz}
\usetikzlibrary{arrows}
\usetikzlibrary{decorations.markings}
\usetikzlibrary{tikzmark,decorations.pathreplacing}
\usetikzlibrary{shapes.geometric}
\numberwithin{equation}{section}
\usepackage{float}
\makeatletter
\def\Ddots{\mathinner{\mkern1mu\raise\p@
\vbox{\kern7\p@\hbox{.}}\mkern2mu
\raise4\p@\hbox{.}\mkern2mu\raise7\p@\hbox{.}\mkern1mu}}
\makeatother

\newcommand{\ds}{\displaystyle}

\DeclareMathOperator{\OSPT}{OSP}
\DeclareMathOperator{\SPT}{SP}

\DeclareMathOperator{\SSYT}{SSYT}

\renewcommand{\sp}{\mathrm{sp}}
\DeclareMathOperator{\s}{\mathrm{s}}
\renewcommand{\s}{\mathrm{s}}

\DeclareMathOperator{\sgn}{sgn}

\newcommand{\x}{\bar{x}}

\DeclareMathOperator{\hs}{hs}
\DeclareMathOperator{\wt}{wt}
\DeclareMathOperator{\spo}{spo}

\newtheorem{thm}{Theorem}[section]

\newtheorem{cor}[thm]{Corollary}
\newtheorem{lem}[thm]{Lemma}
\newtheorem{defn}[thm]{Definition}
\newtheorem{rem}[thm]{Remark}

\newtheorem{eg}[thm]{Example}

\title
{
A determinantal formula for orthosymplectic
Schur functions
}


\author{Nishu Kumari}
\address{Nishu Kumari, Department of Mathematics, 
Indian Institute of Science, Bangalore  560012, India.}
\email{nishukumari@iisc.ac.in}

\date{\today}
\begin{document}
\begin{abstract}
We prove a new determinantal formula for the characters of irreducible representations of  orthosymplectic  Lie  superalgebras analogous to the formula developed by Moens and Jeugt (J. Algebraic Combin., 2003) for general linear Lie superalgebras.
Our proof uses the Jacobi--Trudi type formulas for orthosymplectic characters. As a consequence, we show that 
the odd symplectic characters 
introduced by Proctor (Invent. Math., 1988) 
are the same as the 
orthosymplectic characters with some specialized indeterminates. 
We also give a generalization of an odd symplectic character identity due to Brent, Krattenthaler and Warnaar (J. Combin. Theory Ser. A, 2016).
\end{abstract}
\maketitle

\section{Introduction}
Hook Schur polynomials (or supersymmetric Schur polynomials) and orthosymplectic Schur polynomials are the characters of irreducible representations of general linear and orthosymplectic Lie superalgebras respectively. These are supersymmetric functions in two sets of commuting variables indexed by integer partitions.
It is well known that hook Schur polynomials form {a} basis for the ring of supersymmetric functions.
For background on the representation theory of Lie superalgebras, see~\cite{cheng2012dualities}.

Hook Schur polynomials were introduced by Berele and Regev~\cite{berele1987hook} combinatorially using hybrid tableaux involving a column-strict part and a row-strict part.
Stembridge~\cite{stembridge1985characterization} gave a different characterization for these polynomials and showed that they
form a basis for the ring of supersymmetric functions.
Moens and Jeugt~\cite{moens2003determinantal} proved the following determinantal formula for these polynomials using the representation theory of {the} general linear superalgebra $\mathfrak{gl}(m/n)$. 

\begin{thm}[{\cite[Equation 1.17]{moens2003determinantal}}]
\label{det-hook}
Let $\lambda$ be a partition such that $\lambda_{n+1} \leq m$
and {$k = \min \{i \mid \lambda_i + n + 1 - i \leq m\}.$} {The} hook Schur polynomial 
$\hs_{\lambda}(X;Y)$ is given by
\[\hs_{\lambda}(X;Y)
    =
\frac{(-1)^{mn-n+k-1}}{D} \det \left( \begin{array}{c|c}
   \left( \frac{1}{x_i+y_j}
\right)_{\substack{1 \leq i\leq n \\ 1\leq j\leq m}}  & 
\left(x_i^{\lambda_j+n-m-j}\right)_{\substack{1 \leq i\leq n \\ 1\leq j\leq k-1}} \\\\
\hline \\
 \left(y_j^{\lambda'_i+m-n-i}\right)_{\substack{1 \leq i\leq m-n+k-1 \\ 1\leq j\leq m}} & {\normalfont\text{\huge0}}
\end{array} \right),
\]
where \[
D= \frac{\prod_{1\leq i<j \leq n}(x_i-x_j)
\prod_{1 \leq i<j \leq m}(y_i-y_j)}{\prod_{i=1}^n \prod_{j=1}^m (x_i+y_j)}.
\]
\end{thm}

The main result of this article is to prove the  following determinantal formula for orthosymplectic Schur polynomials analogous to the formula in \cref{det-hook} for hook Schur polynomials. 

\begin{thm} 
\label{thm:main}
Let $\lambda$ be a partition of length at most $n$ and {$k = \min \{j \mid \lambda_j + n + 1 - j \leq m\}.$} {The} orthosymplectic Schur function $\spo_{\lambda}(X;Y)$ is given by
\begin{equation}
\label{det-ortho}
    \spo_{\lambda}(X;Y)
    =
\frac{(-1)^{mn-n+k-1}}{D} 
\det \left( \begin{array}{c|c}
   R  & X_{\lambda} \\[-0.3cm]\\\hline\\[-0.3cm]
   Y_{\lambda}  & 0
\end{array} \right)
\end{equation}
where 
\[
R= \left( \frac{x_i}
{(x_i+y_j)\prod_{q=1}^m (\x_i + y_q) }
-
\frac{\x_i}
{(\x_i+y_j)\prod_{q=1}^m (x_i + y_q) }
\right)_{\substack{1 \leq i\leq n \\ 1\leq j\leq m}},\]
\[
X_{\lambda} = \left(
\frac{x_i^{\lambda_j+n-m-j+1}}
{\prod_{q=1}^m (\x_i + y_q)}-
\frac{\x_i^{\lambda_j+n-m-j+1}}
{\prod_{q=1}^m (x_i+y_q)}
\right)_{\substack{1 \leq i\leq n \\ 1\leq j\leq k-1}},  
\qquad 
Y_{\lambda} = \left(
y_j^{\lambda'_i+m-n-i}
\right)_{\substack{1 \leq i\leq m-n+k-1 \\ 1\leq j\leq m}}\]
and 
\begin{equation*}
D= \frac
{\prod_{i=1}^{n}
    (x_i-\x_i)
    \prod_{1 \leq i<j\leq n} 
(x_i+\x_i-x_j-\x_j)
\prod_{1 \leq i<j\leq m} (y_i-y_j)
}{\prod_{i=1}^n \prod_{j=1}^m (x_i+y_j) (\x_i+y_j)}.    
\end{equation*}
\end{thm} 

\begin{eg} Let $n=1$, $m=2$ and consider the partition
    $\lambda=(2)$. 
    Then $k=2$ and by \eqref{det-ortho}, we have 
    \begin{multline*}
\spo_{\lambda}(x_1;y_1,y_2) = \frac{(x_1+y_1)(x_1+y_2)(\x_1+y_1)(\x_1+y_2)}{(x_1-\x_1)(y_1-y_2)} \\
\times
   \left( \frac{x_1}{(\x_1+y_1)(\x_1+y_2)}-
       \frac{\x_1}{(x_1+y_1)(x_1+y_2)} \right)
    \times \det 
    \left(
    \begin{array}{cc}
       y_1 & y_2\\
      1 & 1 
    \end{array}
    \right).        
    \end{multline*}
    It turns out to be equal to
    \begin{equation}
    \label{symplec}
   x_1^2+\x_1^2+1+y_1(x_1+\x_1)+y_2(x_1+\x_1)+y_1y_2.     
    \end{equation}
    {As the expression for $\spo_{(2)}(x_1;y_1,y_2)$ in \eqref{symplec} is the same as the one given in \eqref{symplec-1}, \cref{thm:main} is verified in this instance.} 
\end{eg}

\begin{cor}
\label{cor:m=1} 
  Let $\lambda$ be a partition of length at most $n$. Then
 \[
\spo_{\lambda}(x_1,x_2,\dots,x_n;y) =
\frac{1}{D}
\det \left(
  x_i^{\lambda_j+n-j+1}-\x_i^{\lambda_j+n-j+1}+
  y \left(x_i^{\lambda_j+n-j} - \x_i^{\lambda_j+n-j}\right) 
\right),
\]   
where $D= 
\prod_{i=1}^{n}
    (x_i-\x_i)
    \prod_{1 \leq i<j\leq n} 
(x_i+\x_i-x_j-\x_j)$. 
\end{cor}

We prove \cref{thm:main} in \cref{sec:ortho} after giving all required definitions in \cref{sec:prelim}. Our proof uses the Jacobi--Trudi type formula for orthosymplectic Schur polynomials.
The Jacobi--Trudi type formula 
was given by Balantekin and Bars~\cite{balantekin1981dimension}. 
It was proved using a different approach by Benkart,  Shader  and  Ram~\cite{benkart1998tensor} and   
recently, Stokke and Visentin~\cite{stokke2016lattice} {proved} the formula using lattice paths. Orthosymplectic Cauchy identities and  {a} Pieri rule for orthosymplectic Schur polynomials were proved in~\cite{orthosymplecticcauchy} and
~\cite{stokke2018orthosymplectic} respectively.

We also show that 
the odd symplectic characters
introduced by Proctor~\cite{Proctor1988OddSG} are the same as the orthosymplectic Schur polynomials with some specialized indeterminates (see \cref{cor: odd}). 
Lastly, we give a generalization of an odd symplectic character identity due to Brent, Krattenthaler and Warnaar, which was found in the study of discrete Mehta-type integrals (see ~\cite{brent2016discrete, talkslides, Okada2019ABF}) in \cref{sec:bkw}.

\section{Preliminaries}
\label{sec:prelim}
 Recall that a {\em partition} $\lambda=(\lambda_1,\lambda_2,\dots)$
 is a sequence of weakly decreasing nonnegative
integers 
containing only finitely many non-zero terms.  
The {\em length} of a partition $\lambda$, which
is the number of positive parts, is denoted by $\ell(\lambda)$. A partition $\lambda$ can be represented pictorially as a {\em Young diagram}, whose {$i^{\text{th}}$} row contains {$\lambda_i$}
 left-justified boxes. We will use the so-called English notation where the first row is on top. 
 The {\em conjugate} of a  partition $\lambda$, denoted $\lambda'$, is the partition whose Young diagram is obtained by transposing the Young diagram of $\lambda$.
 For example, \cref{fig:my_label} shows the Young diagrams of $\lambda=(3,2,2)$ and its conjugate $\lambda'=(3,3,1)$.
\begin{figure}[H]
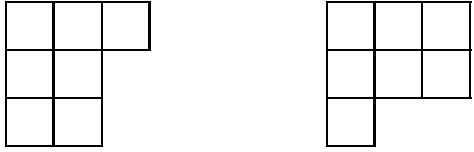

    \centering
    \begin{subfigure}[b]{0.25\textwidth}
         \centering
    {\ydiagram{3,2,2}} 
    \end{subfigure}
    \begin{subfigure}[b]{0.25\textwidth}
         \centering
{\ydiagram{3,3,1}}
\end{subfigure}
    \caption{{{The Young diagram of $\lambda=(3,2,2)$  and $\lambda'=(3,3,1)$}}}
    \label{fig:my_label}
\end{figure}

Throughout, we fix $n$ and $m$ to be positive integers. Let $X=(x_1,x_2,\dots,x_n)$ and $Y=(y_1,\dots,y_m)$  be tuples of commuting indeterminates. Define $\x=1/x$ for an indeterminate $x$ and write $\overline{X}=(\x_1,\dots,\x_n)$. We use the shorthand notation $[a]$ to denote the set $\{1,2,\dots,a\}$ for $a \in \mathbb{Z}_{+}$.

We consider the ring $\mathbb{Z}[x_1,\dots,x_n]$ of polynomials in $n$ commuting variables $x_1,\dots,x_n$ with integer coefficients. A polynomial in this ring is {\em symmetric} if it is invariant under the action of permuting the variables. 
The {\em elementary symmetric function} $e_{\lambda}(X)$ and the {\em complete symmetric function} $h_{\lambda}(X)$ indexed by a partition $\lambda=(\lambda_1,\lambda_2,\dots,\lambda_n)$ are defined as:
\begin{equation}
\label{def:ele}
e_{\lambda}(X) = \prod_{i=1}^n e_{\lambda_i}(X), 
\end{equation}
where
\begin{equation*}
e_r(X) =
\begin{cases}
   \ds \sum_{1 \leq i_1 < i_2 < \dots < i_r \leq n} x_{i_1} x_{i_1} \dots x_{i_r} &  r \geq 1,\\
   \qquad  \qquad  \qquad  \qquad 1 & r=0,
\end{cases}
\end{equation*}
and 
\begin{equation}
\label{def:cmp}
h_{\lambda}(X) = \prod_{i=1}^n h_{\lambda_i}(X),
\end{equation}
where
\begin{equation*}
h_r(X) =
\begin{cases}
   \ds \sum_{1 \leq i_1 \leq i_2 \leq \dots \leq i_r \leq n} x_{i_1} x_{i_1} \dots x_{i_r} &  r \geq 1,\\
   \qquad  \qquad  \qquad  \qquad 1 & r=0.
\end{cases}  
\end{equation*}
Also, define $e_r(X)=h_r(X)=0$ for $r<0$.

For a partition $\lambda=(\lambda_1,\ldots,\lambda_n)$ of length at most $n$, the \emph{Schur polynomial} is given by the following determinantal formula:
\begin{equation}
\label{gldef}
\s_\lambda(X)=\frac{\displaystyle \det
\Bigl(x_i^{\lambda_j+n-j}\Bigr)_{1\le i,j\le n}}
{\displaystyle\det \Bigl(x_i^{n - j}\Bigr)_{1\le i,j\le n}}.
\end{equation}
The expression in \eqref{gldef} is essentially the Weyl character formula.

For partitions $\lambda$ and $\mu$ such that $\mu \subset \lambda \, (\mu_i \leq \lambda_i$, for all $i \geq 1)$, the {\emph{skew
shape} $\lambda/\mu$} is the set theoretic difference where
{$\lambda$ and $\mu$ are regarded as Young diagrams, i.e., collections of boxes}. 
Recall that 
a {{\em semistandard Young tableau}} of shape $\lambda/\mu$ is a filling of $\lambda/\mu$ such that
\begin{itemize}
    \item the entries increase weakly {across} rows,
    \item the entries increase strictly {down} columns.
\end{itemize} 
The {\em skew Schur function} $\s_{\lambda/\mu}(X)$ is given by 
\begin{equation}
\label{def:tab-skew-schur}
\s_{\lambda/\mu}(X)=\sum_{T} \wt(T), \qquad \wt(T) \coloneqq \prod_{i = 1}^n x_i^{c_i(T)},    
\end{equation}
where the sum is taken over all semistandard Young tableaux of shape $\lambda/\mu$ with entries from the set $[n]$ and $c_i(T)$, $i \in [n]$ is the number of occurrences of $i$ in $T$. {If $\mu=\emptyset$, then $\s_{\lambda/\mu}(X)=\s_{\lambda}(X)$.}
The following Jacobi--Trudi formula expresses $\s_{\lambda/\mu}(X)$ in terms of complete symmetric functions:
\begin{equation}
\label{def:S-eJccob}
    \s_{\lambda/\mu}(X)=\det \left(h_{\lambda_i-\mu_j-i+j}(X)\right)_{1 \leq i,j \leq r},
\end{equation}   
where $r$ is any integer such that $r \geq \ell(\lambda)$.

We now give a supersymmetric analogue of the symmetric functions.
We consider the ring $\mathbb{Z}[x_1,\dots,x_n,y_1,\dots,y_m]$ of polynomials in $(n+m)$ commuting variables $x_1,\dots,x_n,$ $y_1,\dots,y_m$ with integer coefficients.
A polynomial $f(X,Y)$ in this ring is {\em doubly symmetric} if it is symmetric in both $(x_1,\dots,x_n)$ and $(y_1,\dots,y_m)$. Moreover, if substituting $x_n=t$ and $y_m=-t$ {results in an element
in $\mathbb{Z}[x_1,\dots,x_{n-1},y_1,\dots,y_{m-1}]$},
then we call $f(X,Y)$ a {\em supersymmetric function}.
The {\em elementary supersymmetric function} and the {\em complete supersymmetric function} indexed by $\lambda=(\lambda_1,\lambda_2,\dots,\lambda_n)$ are:
\begin{equation}
\label{def:superele}
 E_{\lambda}(X;Y) = \prod_{i=1}^n E_{\lambda_i}(X;Y), 
 \end{equation}
 where
 \begin{equation*}
 E_r(X;Y) = 
 \begin{cases}
     \ds \sum_{j=0}^{r} e_{j}(X) h_{r-j}(Y) &
     r \geq 1,\\
   \qquad \qquad  1 & r=0,
 \end{cases}
\end{equation*}
and 
\begin{equation}
\label{def:sS}
 H_{\lambda}(X;Y) = 
 \prod_{i=1}^n H_{\lambda_i}(X;Y),
 \end{equation}
 where 
 \begin{equation*}
 H_r(X;Y) = 
 \begin{cases}
     \ds \sum_{j=0}^{r} h_{j}(X) e_{r-j}(Y) &
     r \geq 1,\\
   \qquad \qquad  1 & r=0.
 \end{cases}
\end{equation*}
It is convenient to define $E_r(X;Y)$ and $H_r(X;Y)$ 
to be
zero for $r<0$. 

\begin{defn}
\label{def:super}
A {\em supertableau} $T$ of shape $\lambda$ with entries 
\[
1 < 2 < \cdots < n < 1' < 2' < \cdots < m'
\]
is a filling of the shape with these entries satisfying the following conditions:
\begin{itemize}
    \item the entries increase weakly {across rows and down columns},
    \item the unprimed entries strictly increase {down} columns,
    \item the primed entries strictly increase {across} rows.
\end{itemize}
\end{defn}
The {\em weight} of such a supertableau is given by
\[
\wt(T) \coloneqq \prod_{i=1}^n x_i^{n_i(T)} 
{\prod_{j=1}^{m} y_j^{n_{j'}(T)}},
\]
where $n_{\alpha}(T)$, for $\alpha$ being any of the entries given in \cref{def:super}, is the number of occurrences of $\alpha$ in $T$. 
Denote the set of supertableaux of skew shape $\lambda$ with entries given in \cref{def:super} 
by $\SSYT_{n/m}(\lambda)$.
For an integer partition $\lambda$, the {\em hook Schur 
(or supersymmetric Schur) 
polynomial}, denoted $\hs_{\lambda}(X;Y)$ is given by
\begin{equation}
\label{def:tab-super-schur}
 \hs_{\lambda/\mu}(X;Y) \coloneqq \sum_{T \, \in \, \SSYT_{n/m}(\lambda)} \wt(T).
 \end{equation}
 \begin{eg}
 Let $n=2$ and $m=1$ and consider the partition $(2,1)$. Then we have
the following supertableaux in $\SSYT_{2/1}((2,1))$:
\[
\begin{ytableau}
\none & 1 & 1\\
\none & 2 & \none
\end{ytableau} 
\begin{ytableau}
\none & 1 & 2\\
\none & 2 & \none
\end{ytableau} 
\begin{ytableau}
\none & 1 & 1\\
\none & 1' & \none
\end{ytableau} 
\begin{ytableau}
\none & 1 & 2\\
\none & 1' & \none
\end{ytableau} 
\]
\[
\begin{ytableau}
\none & 1 & 1'\\
\none & 2 & \none 
\end{ytableau} 
\begin{ytableau}
\none & 2 & 2\\
\none & 1' & \none
\end{ytableau} 
\begin{ytableau}
\none & 1 & 1'\\
\none & 1' & \none
\end{ytableau} 
\begin{ytableau}
\none & 2 & 1'\\
\none & 1' & \none
\end{ytableau} 
\]
Therefore,
\[
 \hs_{(2,1)}(x_1,x_2;y_1)=x_1^2x_2 + x_1 x_2^2 + x_1^2 y_1 + 2x_1x_2y_1+ x_2^2y_1+x_1y_1^2+x_2y_1^2.
\]
 \end{eg}
\noindent 
The Jacobi--Trudi formula expresses $\hs_{\lambda}(X;Y)$ in terms of
the complete supersymmetric functions~\cite{pragacz1992jacobi}. 
\begin{equation}
    \label{sS-jbi}
    \hs_{\lambda}(X;Y)=\det\left(H_{\lambda_i-i+j}(X;Y)\right)_{1 \leq i,j \leq n}.
\end{equation}
Note that the hook Schur polynomial $\hs_{\lambda}(X;Y)$ is nonzero if and only if $\lambda_{n+1} \leq m$.  
Now we consider orthosymplectic Schur functions after defining symplectic Schur functions.

\begin{defn}
\label{defn:st}
{A {\em symplectic tableau} $T$ of shape $\lambda$ is a filling of the shape with entries from 
$
1 < \overline{1} < 2 < \overline{2} < \cdots < n < \overline{n}
$
satisfying the following conditions:}
\begin{itemize}
    \item the entries increase weakly {across rows and strictly down columns},
    \item  the entries in the $i^{\text{th}}$ row must be greater than or equal to $i$.
\end{itemize}
\end{defn}
\noindent 
The weight of such a tableau is given by
\[
\wt(T)= {\prod_{i=1}^n x_i^{n_{i}(T)-n_{\overline{i}}(T)},}
\]
where $n_i(T)$ and $n_{\overline{i}}(T)$ {are} the number of occurrences of $i$ and $\overline{i}$ in $T$ respectively. Denote the set of symplectic tableaux of shape $\lambda$ filled with entries given in \cref{defn:st} by {$\SPT_n(\lambda)$}. 
\begin{eg}
The following is a symplectic tableau $T$ of shape $\lambda=(3,2,2)$ with $\wt(T)={x_2^{-1}}$.
\[
\begin{ytableau}
  1 & \overline{1} & 2 \\
    \overline{2} & \overline{2} \\
    4 & \overline{4}
\end{ytableau}
\]
\end{eg}
\noindent
The {\em symplectic Schur function} is defined as
\[
\sp_{\lambda}(X) \coloneqq \sum_{T \in \SPT_n(\lambda)} \wt(T),
\]
where the sum runs over the set of tableaux in $\SPT_n(\lambda)$. The symplectic Schur function can be expressed using the Weyl character formula in the following way. Let $\lambda$ be a partition of length at most $n$. Then
\begin{equation}
\label{spdef}
\sp_\lambda(X)=
\frac{\ds \det \Bigr(x_i^{\lambda_j+n-j+1}-\x_i^{\lambda_j+n-j+1}\Bigr)_{1\le i,j\le n}}
{\ds \det
\Bigr(x_i^{n-j+1}-\x_i^{n-j+1}\Bigr)_{1\le i,j\le n}},
\end{equation}
and the denominator here is
\begin{equation}
\label{spdenom}
\det \Bigr(x_i^{n-j+1}-\x_i^{n-j+1}\Bigr)_{1\le i,j\le n}
= \prod_{i=1}^n(x_i-\x_i)\,\prod_{1\le i<j\le n}(x_i+\x_i-x_j-\x_j).
\end{equation}

\begin{defn}
\label{def:ot}
Let  $B_0=\{1,\overline{1},2,\overline{2},\dots,n,\overline{n}\}$, $B_1=\{1',\dots,m'\}$, $B=B_0 \cup B_1$ and 
 order $B$ as follows:
\[
1 < \overline{1} < 2 < \overline{2} < \cdots < n < \overline{n} < 1' < 2' < \cdots < m'.
\]
An {\em orthosymplectic tableau} $T$ of shape $\lambda$ is a filling of {the} Young diagram of  $\lambda$ with entries from $B$ satisfying the following conditions:
\begin{itemize}
    \item 
     the portion $S$ of $T$ that contains entries only from $B_0$ is symplectic,
    \item  the skew tableau formed by removing S from T is strictly increasing {across} the rows from left to right and weakly increasing {down} columns from top to bottom.
\end{itemize}
\end{defn}
\noindent
The weight of such a tableau is given by \[
\wt(T)= {\prod_{i=1}^n x_i^{n_{i}(T)-n_{\overline{i}}(T)} \prod_{j=1}^m y_j^{n_{i'}(T)}},
\]
where $n_{\alpha}$ is the number of occurrences of $\alpha$, for $\alpha$ being any of the entries in {$B$}.
Denote the set of orthosymplectic tableaux of shape $\lambda$ filled with entries {in $B$} 
by {$\OSPT_{n/m}({\lambda})$}.
\begin{eg}
The following is an orthosymplectic tableau $T$ of shape $\lambda=(3,3,2,1)$ and {$\wt(T)=x_1^{-1} x_4 y_1 y_2 y_3^3$}.
\[
\begin{ytableau}
    \overline{1} & 2 & 3'\\
    \overline{2} & 1' & 3'\\
    4 & 3'\\
    2'
\end{ytableau}
\]
\end{eg}
\noindent
The {\em orthosymplectic Schur function} indexed by a partition $\lambda$ is given by  
\[
\spo_{\lambda}(X;Y) \coloneqq \sum_{T \, \in \, \OSPT_{n/m}({\lambda})} \wt(T).
\]
Note that  
$\spo_{\lambda}(X;\emptyset)=\sp_{\lambda}(X)$. Moreover, 
\[
\spo_{\lambda}(X;Y) = \sum_{\mu} \sp_{\mu}(X) \s_{\lambda'/\mu'}(Y).
\]
\begin{eg}
\label{eg}
Let $n=1$, $m=2$ and consider the partition
    $\lambda=(2)$. {Then we have
the following orthosymplectic tableaux:}
\[
\begin{ytableau}
 1 & 1
\end{ytableau} 
\qquad 
\begin{ytableau}
 \overline{1} & \overline{1}
\end{ytableau} 
\qquad 
\begin{ytableau}
1 & \overline{1}
\end{ytableau} 
\qquad 
\begin{ytableau}
 1 & 1'
\end{ytableau} 
\]
\[
\begin{ytableau}
 \overline{1} & 1'
\end{ytableau} 
\qquad 
\begin{ytableau}
 1 & 2'
\end{ytableau} 
\qquad 
\begin{ytableau}
  \overline{1} & 2'
\end{ytableau} 
\qquad 
\begin{ytableau}
1' & 2'
\end{ytableau} 
\]
Therefore,
\begin{equation}
    \label{symplec-1}
     \spo_{(2)}(x_1;y_1,y_2)=   x_1^2+\x_1^2+1+y_1(x_1+\x_1)+y_2(x_1+\x_1)+y_1y_2.
\end{equation}
\end{eg}
\noindent 
The Jacobi--Trudi type identity for the orthosymplectic Schur functions is given by
\begin{equation}
\label{JT-ortho}
     \spo_{\lambda}(X;Y)=
    { \det \left(
    \begin{array}{c|c} 
    \left(J_{\lambda_i-i+1}\right)_{1 \leq i \leq n} & 
    (J_{\lambda_i-i+j}+J_{\lambda_i-i-j+2})_{\substack{1 \leq i \leq n\\
    2 \leq j \leq n}}  
    \end{array} \right)},
\end{equation}
    where $J_r=\sum_{l=0}^r h_l(X,\overline{X}) e_{r-l}(Y)$. {Here $h_l(X,\overline{X})$ is the Laurent polynomial obtained by substituting $x_{n+i}=\x_i, 1 \leq i \leq n$ in $h_{l}({x_1,x_2,\dots,x_{2n}})$.}

Note that the orthosymplectic Schur polynomial $\spo_{\lambda}(X;Y)$ is nonzero if and only if $\lambda_{n+1} \leq m$. The following remark gives an expression for $\spo_{\lambda}(X;y)$ if 
$\lambda_{n+1} \leq 1$.
\begin{rem}
\label{rem:len}
{For $m=1$, suppose} $\lambda$ is a partition such that $\ell(\lambda)=\ell$ and $\lambda_{n+1} \leq 1$.
If $\ell \leq n$, then $\spo_{\lambda}(X;y)$ is given by \cref{cor:m=1}.
If $\ell>n$, then $\lambda_{n+1}=\cdots=\lambda_{\ell}=1$. So, by \cref{def:ot}, $T$ contains $1'$ in all the boxes of $\lambda$ from the $(n+1)^{\text{th}}$ row to $\ell^{\text{th}}$ row, for all $T \in \OSPT_{n/1}(\lambda)$. Therefore, {using \eqref{spdef} and \eqref{spdenom}} we have
 \[
\spo_{\lambda}(X;y) =
\frac{y^{\ell-n}}{D}
\det \left(
  x_i^{\lambda_j+n-j+1}-\x_i^{\lambda_j+n-j+1}+
  y (x_i^{\lambda_j+n-j} - \x_i^{\lambda_j+n-j}) 
\right),
\]   
where $D= 
\prod_{i=1}^{n}
    (x_i-\x_i)
    \prod_{1 \leq i<j\leq n} 
(x_i+\x_i-x_j-\x_j)$. 
\end{rem}

Now we give a relation between odd symplectic characters defined by Proctor~\cite{Proctor1988OddSG} and orthosymplectic Schur functions after defining odd symplectic characters. 

\begin{defn}
 Let $\lambda$ be a partition of length at most $n$. An \emph{$(n-1,1)$-symplectic
tableau} of shape $\lambda$ is a filling of the Young diagram of $\lambda$ with entries from
\[
1 < \overline{1} < 2 < \overline{2} < \cdots < n-1 < \overline{n-1}  < n 
\]
satisfying the following three conditions:
\begin{itemize}
\item the entries increase weakly {across rows and strictly down columns},
    \item the entries of the $i^{\text{th}}$ row are greater than or equal to $i$.
\end{itemize}
We denote by ${\SPT^{(n-1,1)}(\lambda)}$ the set of $(n-1,1)$-symplectic tableaux of shape $\lambda$. Given {an} $(n-1,1)$-symplectic tableau $T$, we define
\[
\wt(T)=x_n^{\alpha_n(T)} \prod_{i=1}^{n-1} x_i^{\alpha_i(T)-\alpha_{\overline{i}}(T)},
\]   
where $\alpha_i(T)$, $\alpha_{\overline{i}}(T)$ for $i \in [n-1]$, and $\alpha_n(T)$ are {the} number of occurences of $i$, $\overline{i}$ and $n$ respectively.
\end{defn}
\noindent
The {\em odd symplectic character} indexed by a partition of length at most $n$ is given by 

\[
\sp^{(n-1,1)}_{\lambda}(x_1,\dots,x_{n-1}|x_n) = \sum_{T \in \SPT^{(n-1,1)}(\lambda)} \wt(T).
\]

{\begin{eg}
Suppose $n=2$ and $\lambda=(2,1)$. 
Then we have the folowing $(1,1)$-symplectic tableaux of shape $\lambda$:
\[
\begin{ytableau}
 1 & 1\\
 2
\end{ytableau} 
\qquad 
\begin{ytableau}
 1 & \overline{1}\\
 2
\end{ytableau} 
\qquad 
\begin{ytableau}
1 & 2\\
2
\end{ytableau} 
\qquad 
\begin{ytableau}
 \overline{1} & \overline{1}\\
 2
\end{ytableau}
\qquad
\begin{ytableau}
 \overline{1} & 2\\
 2
\end{ytableau}
\]
Therefore,
\[
\sp^{(1,1)}_{(2,1)}(x_1|x_2)=x_1^2x_2+x_2+x_1x_2^2+\x_1^2x_2+\x_1x_2^2.
\]
\end{eg}}

Okada~\cite{Okada2019ABF} proved the following determinantal formula for odd symplectic characters.

\begin{thm}[{\cite[Theorem 1.1]{Okada2019ABF}}]
\label{thm:odd}
    For a partition of length at most $n$, let $A_{\lambda}=(a_{ij})$ be {an} $n \times n$ matrix such that 
\[
a_{ij}= \begin{cases}
 (x_i^{\lambda_j+n-j+1}-\x_i^{\lambda_j+n-j+1})
 -y^{-1}(x_i^{\lambda_j+n-j}-\x_i^{\lambda_j+n-j})
 & 1 \leq i \leq n-1,\\
 y^{\lambda_j+n-j}  
 & i=n.
\end{cases}
\] Then
    \[
    \sp^{(n-1,1)}_{\lambda}(x_1,\dots,x_{n-1}|x_n) = \frac{\det A_{\lambda}}{\det A_{\emptyset}},
    \]  
    where 
    \[\det A_{\emptyset} = \prod_{i=1}^{n-1}
    (x_i-x_i^{-1})
    \prod_{1 \leq i<j\leq n}(x_i+\x_i-x_j-\x_j).
    \]
\end{thm}

\noindent
Comparing the determinantal formulas in \cref{cor:m=1} and \cref{thm:odd}, we get the following result.

\begin{cor}
\label{cor: odd}
Suppose $\lambda$ is a partition of length at most $n$. Then 
    \[
\sp^{(n-1,1)}_{\lambda}(x_1,\dots,x_{n-1}| x_n) =
\spo_{\lambda}(x_1,\dots,x_{n-1},x_n^{-1};-x_n^{-1}).
\]
\end{cor}

\section{Determinant formula for orthosymplectic Schur functions}
\label{sec:ortho}
In this section, we give a proof of \cref{thm:main}.
{We first prove \cref{thm:equiv-ortho}, which is
equivalent to \cref{thm:main}.}
Before that, let us recall the following lemmas.

\begin{lem}[{\cite[Lemma A.54]{FulHar91}}] 
\label{lem:evaluation}
Let 
$\zeta_j(p)=x_j^p-\x_j^p$, $1 \leq j \leq n$. Then for any integer $l \geq 0$, $\zeta_j(l)$ is the product of the $1 \times n$, $n \times n$ and $n \times 1$ matrices
    \[
    \left( 
    \begin{array}{c|c}
      \overline{H}_{l-n} &  \left( \overline{H}_{l-n-1+j}+ \overline{H}_{l-n+1-j} 
      \right)_{2 \leq j \leq n}   
      \end{array}
    \right)
    \Big( 
      (-1)^{v-u} e_{v-u}
    \Big)_{1 \leq u,v \leq n}
    \Big( 
      \zeta_j(n+1-u) 
    \Big)_{1 \leq u \leq n},
    \]
    where $\overline{H}_r=h_r(X,\overline{X})$ for all $r \in \mathbb{Z}$.
\end{lem}
\begin{lem}[{\cite[Chapter I.1, Equation (1.7)]{macdonald-2015}}]
\label{lem:separate}
    Let $\lambda$ be a partition such that $\ell(\lambda) \leq n_1$ and $\lambda_1 \leq n_2$. Then 
    {$\{\lambda_i+n_1-i \mid 1 \leq i \leq n_1\}$ and 
    $\{n_1-1+j-\lambda_j' \mid 1 \leq j \leq n_2\}$}
    form a set partition of $\{0,1,\dots,n_1+n_2-1\}$.
\end{lem}
\begin{thm} 
\label{thm:equiv-ortho}
Let $\lambda$ be a partition of length at most $n$ and {$k = \min \{j \mid \lambda_j + n + 1 - j \leq m\}.$} Then the orthosymplectic Schur function $\spo_{\lambda}(X;Y)$ is given by
\begin{equation}
\label{det-ortho-h}
\spo_{\lambda}(X;Y)
    =
\frac{(-1)^{mn-n+k-1}}{D} \det \left( \begin{array}{c|c}
   R  & X_{\lambda} \\[-0.3cm]\\\hline\\[-0.3cm]
   Y_{\lambda}  & 0
\end{array} \right),
\end{equation}
where
\[
R= \left( (-1)^{m-j} (x_i^{j}-\x_i^{j}) 
\right)_{\substack{1 \leq i\leq n \\ 1\leq j\leq m}}, \quad Y_{\lambda} = \left(h_{\lambda'_i-n-i+j}(Y)\right)_{\substack{1 \leq i\leq m-n+k-1 \\ 1\leq j\leq m}},\]
\[
X_{\lambda} = 
\left(
{x_i^{\lambda_j+n-m-j+1}}{\prod_{q=1}^m (x_i+y_q)}-
{\x_i^{\lambda_j+n-m-j+1}}{\prod_{q=1}^m (\x_i + y_q)} 
\right)_{\substack{1 \leq i\leq n \\ 1\leq j\leq k-1}},
\] $\psi$
and 
\[
D= 
{\prod_{i=1}^{n}
    (x_i-\x_i)
    \prod_{1 \leq i<j\leq n} 
(x_i+\x_i-x_j-\x_j)}.
\]
\end{thm}

\begin{proof}
 By {the} Jacobi--Trudi type identity for orthosymplectic Schur functions \eqref{JT-ortho}, we have   
 \begin{equation}
 \label{1}
 \spo_{\lambda}(X;Y)=
 { \det \left(
    \begin{array}{c|c} 
    \left(J_{\lambda_i-i+1}\right)_{1 \leq i \leq n} & 
    (J_{\lambda_i-i+j}+J_{\lambda_i-i-j+2})_{\substack{1 \leq i \leq n\\
    2 \leq j \leq n}}  
    \end{array} \right)},  
 \end{equation}
where $J_r=\ds \sum_{p=0}^r h_p(X,\overline{X}) e_{r-p}(Y)= \ds \sum_{p=0}^r h_{r-p}(X,\overline{X}) e_{p}(Y)$ for all $r \geq 0$ {and $J_r=0$ otherwise.} Since $e_{p}(Y)=0$ for $p>m$ and $h_{r-p}(X,\overline{X})=0$, for $p>r$, 
\begin{equation}
\label{2}
J_r=\ds \sum_{p=0}^{m} h_{r-p}(X,\overline{X}) e_{p}(Y), \quad r \in \mathbb{Z}.
\end{equation}
Substituting the expression for $J_r$ from \eqref{2} in \eqref{1}, the required orthosymplectic Schur function is given by
    \[ 
   \det \left(
   \begin{array}{c|c}
 \left( \ds \sum_{p=0}^m 
\overline{H}_{\lambda_i-i+1-p} e_p(Y)
\right)_{1 \leq i \leq n}
&
 \left( \ds \sum_{p=0}^m 
(\overline{H}_{\lambda_i-i+j-p}  +
    \overline{H}_{\lambda_i-i-j+2-p}) e_p(Y) \right)_{\substack{1 \leq i \leq n\\
    2 \leq j \leq n}}  
    \end{array}
    \right)
    \]
where $\overline{H}_r=h_r(X,\overline{X})$ for all $r \in \mathbb{Z}$. 
   {Observe that 
    \begin{equation}
    \label{spo}
        \spo_{\lambda}(X;Y)=\det \left(
    \sum_{p=0}^m  e_p(Y) M^{p}
    \right),
    \end{equation}
where \[
    M^{p} =  \left(
   \begin{array}{c|c}
 \left( \ds 
\overline{H}_{\lambda_i-i+1-p} 
\right)_{1 \leq i \leq n}
&
 \left( \ds
\overline{H}_{\lambda_i-i+j-p}  +
    \overline{H}_{\lambda_i-i-j+2-p} \right)_{\substack{1 \leq i \leq n\\
    2 \leq j \leq n}}  
    \end{array}
    \right).
    \]
If $\lambda_i+n-i+1-p \geq 0$, then by \cref{lem:evaluation}, we have     
    \[
    (M^{p}_{i,v})_{1 \leq v \leq n} \left( 
      (-1)^{v-u} e_{v-u}
    \right)_{1 \leq u,v \leq n}
    \left( 
    x_j^{n+1-u}-\x_j^{n+1-u} 
    \right)_{1 \leq u \leq n} = x_j^{\lambda_i+n-i+1-p} -
   \x_j^{\lambda_i+n-i+1-p}
    \]
and 
$M^{p}_{i,v}=0$ for all $1 \leq v \leq n$ otherwise.    
 As a consequence, we have 
 \begin{equation}
\label{M}
 \begin{split}
  \left(\sum_{p=0}^m  e_p(Y)M^{p} \right) & \Big( 
      (-1)^{v-u} e_{v-u}
    \Big)_{1 \leq u,v \leq n} 
    \Big( 
    x_j^{n+1-u}-\x_j^{n+1-u} 
    \Big)_{1 \leq u \leq n} \\
    = & \left( \begin{array}{c}
     \left( \ds \sum_{p=0}^m 
     \left(
    x_j^{\lambda_i+n-i+1-p} -
   \x_j^{\lambda_i+n-i+1-p} 
   \right)
   e_p(Y) \right)_{1 \leq i \leq k-1} \\[-0.3cm]\\\hline\\[-0.3cm]
      \left(
   \ds   \sum_{p=0}^{\lambda_i+n-i} 
      \left(
    x_j^{\lambda_i+n-i+1-p} - \x_j^{\lambda_i+n-i+1-p} \right)
    e_p(Y) \right)_{k \leq i \leq n} 
\end{array} \right)_{n \times 1}.
\end{split}
    \end{equation}
Therefore, using the expressions given in \eqref{spo},\eqref{M}
and the fact that 
\[
\det. \left( 
      (-1)^{v-u} e_{v-u}
    \right)_{1 \leq u,v \leq n} = 1,
    \]
    the orthosymplectic Schur function $\spo_{\lambda}(X;Y)$ is given by  
    \begin{equation}
 \label{(1)}
 \frac{1}{D} \det \left( \begin{array}{c}
     \left( \ds \sum_{p=0}^m 
     \left(
    x_j^{\lambda_i+n-i+1-p} -
   \x_j^{\lambda_i+n-i+1-p} 
   \right)
   e_p(Y) \right)_{1 \leq i \leq k-1} \\[-0.3cm]\\\hline\\[-0.3cm]
      \left(
   \ds   \sum_{p=0}^{\lambda_i+n-i} 
      \left(
    x_j^{\lambda_i+n-i+1-p} - \x_j^{\lambda_i+n-i+1-p} \right)
    e_p(Y) \right)_{k \leq i \leq n} 
\end{array} \right)_{1 \leq j \leq n},    
 \end{equation}   
where 
\[D= \det \Big( 
x_j^{n+1-u}-\x_j^{n+1-u}
    \Big)_{1 \leq u,j \leq n} =
\ds \prod_{i=1}^{n}
    (x_i-\x_i)
    \prod_{1 \leq i<j\leq n} 
(x_i+\x_i-x_j-\x_j).\]}
 Note that
 \begin{equation}
 \label{(2)}
    \sum_{p=0}^m 
    x_i^{-p} e_p(Y) = \prod_{j=1}^m (1+\x_i y_j),
    \qquad
    \sum_{p=0}^m 
    x_i^{p} e_p(Y) = \prod_{i=1}^m (1+x_i y_j). 
 \end{equation}
 Consider the permutation {$\sigma=(k,\dots,n,1,\dots,k-1) \in$}
 $S_n$. The number of inversions of $\sigma$ is $(k-1)(n-k+1)$.
{Permuting the rows of the determinant in \eqref{(1)} by $\sigma$ and then transposing the determinant, we have}
  \begin{equation}
  \label{(3)}
      \spo_{\lambda}(X;Y)=\frac{\sgn(\sigma)}{D} \det
\left(
\begin{array}{c|c}
    (A_{i,j})_{\substack{1 \leq i \leq n\\
    k \leq j \leq n}}  & (B_{i,j})_{\substack{1 \leq i \leq n\\
    1 \leq j \leq k-1}} 
\end{array}
\right),
  \end{equation} 
where 
\[
A_{i,j}= \sum_{p=0}^{\lambda_j+n-j} 
      \left(
    x_i^{\lambda_j+n-j+1-p} - \x_i^{\lambda_j+n-j+1-p} \right)
    e_p(Y)
\]
and
\begin{multline*}
 B_{i,j}= \sum_{p=0}^{m} 
      \left(
    x_i^{\lambda_j+n-j+1-p} - \x_i^{\lambda_j+n-j+1-p} \right)
    e_p(Y)
  \\  = 
    x_i^{\lambda_j+n-j+1}\left(\prod_{q=1}^m (1+\x_i y_q)\right)-\x_i^{\lambda_j+n-j+1}\left(\prod_{q=1}^m (1+x_i y_q)\right),   
\end{multline*}
where the last equality uses the relations given in \eqref{(2)}.
Let 
{$\Lambda \coloneqq \{\lambda_j+n-j+1 \mid k \leq j \leq n\}$, $\Lambda' \coloneqq \{n+i-\lambda_i'\mid 1 \leq i \leq m-n+k-1\}$ and  
\[
M \coloneqq
\left(
\begin{array}{c|c|c}
((-1)^{m-j}(x_i^j-\x_i^j))_{\substack{1 \leq i \leq n\\
    j \, \in \, \Lambda'}}
   & (A_{i,j})_{\substack{1 \leq i \leq n\\
    k \leq j \leq n}}  & 
    (B_{i,j})_{\substack{1 \leq i \leq n\\
    1 \leq j \leq k-1}} \\
    &&\\\hline&\\
   (h_{-i+j}(Y))_{\substack{
   i \, \in \, \Lambda'\\
    j \, \in \, \Lambda'}}
    & {\normalfont\text{\huge0}} & {\normalfont\text{\huge0}}
\end{array}
\right).
\]
Since
\[\det (h_{-i+j}(Y))_{\substack{
   i \, \in \, \Lambda'\\ 
    j \, \in \, \Lambda'}} = 1,
    \]   
    \[
    \det M= (-1)^{n(m - n + k - 1)} \det
\left(
\begin{array}{c|c}
    (A_{i,j})_{\substack{1 \leq i \leq n\\
    k \leq j \leq n}}  & (B_{i,j})_{\substack{1 \leq i \leq n\\
    1 \leq j \leq k-1}} 
\end{array}
\right).
    \]}
So, we can rewrite \eqref{(3)} 
in the following manner:
\[
\spo_{\lambda}(X;Y)=\frac{(-1)^b}{D} \det
\left(
\begin{array}{c|c|c}
((-1)^{m-j}(x_i^j-\x_i^j))_{\substack{1 \leq i \leq n\\
    j \, \in \, \Lambda'}}
   & (A_{i,j})_{\substack{1 \leq i \leq n\\
    k \leq j \leq n}}  & 
    (B_{i,j})_{\substack{1 \leq i \leq n\\
    1 \leq j \leq k-1}} \\
    &&\\\hline&\\
   (h_{-i+j}(Y))_{\substack{
   i \, \in \, \Lambda'\\
    j \, \in \, \Lambda'}}
    & {\normalfont\text{\huge0}} & {\normalfont\text{\huge0}}
\end{array}
\right),
\]  
where $b=(k-1)(n-k+1)+n(m-n+k-1)=mn-(n-k+1)^2$. 
Observe that 
\[
(A_{i,j})_{\substack{1 \leq i \leq n\\
    k \leq j \leq n}} = \left(
    \sum_{p=0}^{j-1}(x_i^{j-p}-\x_i^{j-p})e_p(Y)
    \right)_{\substack{1 \leq i \leq n\\
    j \, \in \, \Lambda}}
\]
So, the orthosymplectic Schur function $\spo_{\lambda}(X;Y)$ is given by 
\begin{equation}
\label{aa}
    \frac{(-1)^b}{D} \det
\left(
\begin{array}{c|c|c}
\left((-1)^{m-j}\phi_i(j)\right)_{\substack{1 \leq i \leq n\\
    j \, \in \, \Lambda'}}
   & 
   \left(
    \sum_{p=0}^{j-1} \phi_i(j-p) e_p(Y)
    \right)_{\substack{1 \leq i \leq n\\
    j \, \in \, \Lambda}}
   & 
    (B_{i,j})_{\substack{1 \leq i \leq n\\
    1 \leq j \leq k-1}} \\
    &&\\\hline&\\
   (h_{-i+j}(Y))_{\substack{
   i \, \in \, \Lambda'\\
    j \, \in \, \Lambda'}}
    & {\normalfont\text{\huge0}} & {\normalfont\text{\huge0}}
\end{array}
\right),
\end{equation}
where $\phi_i(j)=x_i^j-\x_i^j$. 
{Note that $\{\Lambda, \Lambda'\}$ is a set partition of $\{1,2,\dots,m\}$ by \cref{lem:separate}. 
Applying the elementary column transformation} 
$C_{\lambda_n+1} \rightarrow C_{\lambda_n+1} - \sum_{p=1}^{\lambda_n} (-1)^{m-p} e_{\lambda_n+1-p}(Y) C_{p}$ and using 
\[
- \sum_{p=1}^{\lambda_n} (-1)^{m-p} e_{\lambda_n+1-p}(Y)  h_{-i+p}(Y)  = (-1)^{m-\lambda_n+1} h_{-i+\lambda_n+1}(Y),
\]
the {expression} in \eqref{aa} is the same as
\[
{\footnotesize
\frac{(-1)^b}{D} 
\left|
\begin{array}{c|c|c|c}
\left(\hspace{-0.025in}(-1)^{m-j}\phi_i(j)\hspace{-0.025in}\right)_{\substack{1 \leq i \leq n\\
    j \, \in \, \Lambda'}}
   & 
   \left(
    \ds \sum_{p=0}^{j-1} \phi_i(j-p) e_p(Y)
    \hspace{-0.05in} \right)_{\substack{1 \leq i \leq n\\
    j \, \in \, \Lambda\\ j \neq \lambda_n+1}}
    &
   (\phi_i(j))_{\substack{1 \leq i \leq n\\j=\lambda_n+1}}
   & 
    (B_{i,j})_{\substack{1 \leq i \leq n\\
    1 \leq j \leq k-1}} \\
    &&&\\\hline&&\\
   (h_{-i+j}(Y))_{\substack{
   i \, \in \, \Lambda'\\
    j \, \in \, \Lambda'}}
    & {\normalfont\text{\huge0}} & 
    \left(\hspace{-0.025in}(-1)^{m-j} h_{-i+j}(Y) \hspace{-0.025in}\right)_{\substack{i \in \Lambda'\\j=\lambda_n+1}}
    &
    {\normalfont\text{\huge0}}
\end{array}
\right|.
}
\]
Successively applying the {elementary column transformations} $C_j \rightarrow C_j-\sum_{p=1}^{j-1} (-1)^{m-p} e_{j-p}C_p$, for $j \in \Lambda$, we have 
\[
\spo_{\lambda}(X;Y)=\frac{(-1)^b}{D} \det
\left(
\begin{array}{c|c|c}
\left((-1)^{m-j}\phi(j)\right)_{\substack{1 \leq i \leq n\\
    j \, \in \, \Lambda'}}
    &
   (\phi(j))_{\substack{1 \leq i \leq n\\ 
    j \, \in \, \Lambda}}
   & 
    (B_{i,j})_{\substack{1 \leq i \leq n\\
    1 \leq j \leq k-1}} \\
    &&\\\hline&\\
   (h_{-i+j}(Y))_{\substack{
   i \, \in \, \Lambda'\\
    j \, \in \, \Lambda'}}
     & 
    ((-1)^{m-j} h_{-i+j}(Y))_{\substack{
   i \, \in \, \Lambda'\\
    j \, \in \, \Lambda}}
    &
    {\normalfont\text{\huge0}}
\end{array}
\right). 	
\]
Taking $(-1)^{m-j}$ out from the $j^{\text{th}}$ column for each $j \in \Lambda$, the orthosymplectic Schur function is given by 
\[
\spo_{\lambda}(X;Y)=\frac{(-1)^{b'}}{D} \det
\left(
\begin{array}{c|c|c}
\left((-1)^{m-j}\phi(j)\right)_{\substack{1 \leq i \leq n\\
    j \, \in \, \Lambda'}}
    &
   ((-1)^{m-j} \phi(j))_{\substack{1 \leq i \leq n\\ 
    j \, \in \, \Lambda}}
   & 
    (B_{i,j})_{\substack{1 \leq i \leq n\\
    1 \leq j \leq k-1}} \\
    &&\\\hline&\\
   (h_{-i+j}(Y))_{\substack{
   i \, \in \, \Lambda'\\
    j \, \in \, \Lambda'}}
     & 
    (h_{-i+j}(Y))_{\substack{
   i \, \in \, \Lambda'\\
    j \, \in \, \Lambda}}
    &
    {\normalfont\text{\huge0}}
\end{array}
\right). 	
\]
where $b'=b+\sum_{j \in \Lambda} (m-j)$. Rearranging the columns, we have 
\[
\spo_{\lambda}(X;Y) = \frac{(-1)^a}{D}  \det \left( \begin{array}{c|c}
\left( (-1)^{m-j} (x_i^{j}-\x_i^j) 
\right)_{\substack{1 \leq i\leq n \\ 1 \leq j \leq m}} 
&
(B_{i,j})_{\substack{1 \leq i \leq n\\
    1 \leq j \leq k-1}}
\\[1.75em]
\hline\\
 \left(h_{\lambda'_i-n-i+j}(Y)
\right)_{\substack{1 \leq i\leq m-n+k-1 \\ 1 \leq j\leq m}} 
&
{\normalfont\text{\huge0}}
\end{array} \right),
\]
where $a=mn-n+k-1$. This completes the proof. 
\end{proof} 
Now we prove that the expressions for $\spo_{\lambda}(X;Y)$ in \cref{thm:main} and \cref{thm:equiv-ortho} are the same. 
\begin{proof}[{Proof of \cref{thm:main}}]  
Let $a=mn-n+k-1$. {Consider the following expression given in \eqref{det-ortho}.}
\begin{equation*}
 {A} \coloneqq  \frac{(-1)^{a}}{D_s} 
 \det \left( 
 \begin{array}{c|c}
   \left( \ds \frac{x_i}
{(x_i+y_j)\prod_{q=1}^m (\x_i + y_q) }
-
\ds \frac{\x_i}
{(\x_i+y_j)\prod_{q=1}^m (x_i + y_q) }
\right)_{\substack{1 \leq i\leq n \\ 1\leq j\leq m}} & X_{\lambda}
\\[1.75em]
\hline\\
  \left(y_j^{\lambda'_i+m-n-i}\right)_{\substack{1 \leq i\leq m-n+k-1 \\ 1\leq j\leq m}} & {\normalfont\text{\huge0}}
\end{array} \right),
\end{equation*}
where
\[
X_{\lambda} = \left(
\frac{x_i^{\lambda_j+n-m-j+1}}
{\prod_{q=1}^m (\x_i + y_q)}-
\frac{\x_i^{\lambda_j+n-m-j+1}}
{\prod_{q=1}^m (x_i+y_q)}
\right)_{\substack{1 \leq i\leq n \\ 1\leq j\leq k-1}} \]
and 
\[
D_s= \frac
{\prod_{i=1}^{n}
    (x_i-\x_i)
    \prod_{1 \leq i<j\leq n} 
(x_i+\x_i-x_j-\x_j)
\prod_{1 \leq i<j\leq m} (y_i-y_j)
}{\prod_{i=1}^n \prod_{j=1}^m (x_i+y_j) (\x_i+y_j)}.
\]
Applying the elementary row transformations $R_i \rightarrow \prod_{q=1}^m (x_i+y_q)(\x_i+y_q) R_i$, for $i \in [1,n]$ and cancelling the term $\prod_{i=1}^n \prod_{j=1}^m (x_i+y_j) (\x_i+y_j)$ from the denominator of $D_s$,  
the expression for {$A$ can be written as}
\[
\frac{(-1)^{a}}
{D_s^1} 
\det \left( \begin{array}{c|c}
   \left( 
   x_i \prod_{\substack{q=1\\ q \neq j}}^{m} 
   (x_i+y_q)
   -\x_i \prod_{\substack{q=1\\ q \neq j}}^{m}  (\x_i+y_q)
\right)_{\substack{1 \leq i\leq n \\ 1\leq j\leq m}}  
& 
 Z_{\lambda} \\[1.75em]
\hline\\
  \left(y_j^{\lambda'_i+m-n-i}\right)_{\substack{1 \leq i\leq m-n+k-1 \\ 1\leq j\leq m}} & {\normalfont\text{\huge0}}
\end{array} \right),
\]
where 
\[
D_s^1 = \prod_{i=1}^{n}
    (x_i-\x_i)
    \prod_{1 \leq i<j\leq n} 
(x_i+\x_i-x_j-\x_j)
\prod_{1 \leq i<j\leq m} (y_i-y_j)
\]
and 
\[
Z_{\lambda} = 
\left(
{x_i^{\lambda_j+n-m-j+1}}{\prod_{q=1}^m (x_i+y_q)}-
{\x_i^{\lambda_j+n-m-j+1}}{\prod_{q=1}^m (\x_i + y_q)} 
\right)_{\substack{1 \leq i\leq n \\ 1\leq j\leq k-1}}.
\]
Let 
\[
M^{k}_{i,j}=x_i
\prod_{\substack{q=1\\ q \neq j}}^{k} 
   (x_i+y_q)
   -\x_i 
   \prod_{\substack{q=1\\ q \neq j}}^{k}  (\x_i+y_q).
\]
 Applying the elementary column transformations $C_j \rightarrow C_j-C_m$, $j \in [1,m-1]$ and then using $y_j^p-y_m^p=(y_j-y_m) h_{p-1}(y_j,y_m)$ to factor out the term $y_j-y_m$ for $j \in [1,m-1]$,
$A$ is given by  
 \[
 \frac{(-1)^{a}}{D_s^2} \det \left( \begin{array}{c|c|c}
   \left(- M^{m-1}_{i,j}
\right)_{\substack{1 \leq i\leq n \\ 1\leq j\leq m-1}} 
& 
\left( M^{m}_{i,j}
\right)_{1 \leq i\leq n}
&  Z_{\lambda}
\\&\\\hline&\\
  \left(h_{\lambda'_i+m-n-i-1}(y_j,y_m)\right)_{\substack{1 \leq i\leq m-n+k-1 \\ 1\leq j\leq m-1}} 
  &  \left(h_{\lambda'_i+m-n-i}(y_m)\right)_{1 \leq i\leq m-n+k-1 }
  & {\normalfont\text{\huge0}}
\end{array} \right), 
  \]
  where 
  \[
  D_s^2 = \prod_{i=1}^{n}
    (x_i-\x_i)
    \prod_{1 \leq i<j\leq n} 
(x_i+\x_i-x_j-\x_j)
\prod_{1 \leq i<j\leq m-1} (y_i-y_j).
  \]
  Now applying the elementary column transformations $C_j \rightarrow C_j-C_{s}$, 
  $j \in [1,s-1]$ and 
  using 
  $
  h_p(y_j,y_{s+1},y_{s+2},\dots,y_m)-h_p(y_s,y_{s+1},y_{s+2},\dots,y_m)=(y_j-y_s) 
  h_{p-1}(y_j,y_s,y_{s+1},y_{s+2},\dots,y_m),
  $
    successively for $s=m-1,m-2,\dots,2$, we have 
\begin{equation*}
    A= \frac{(-1)^{a}}{D} \det \left( \begin{array}{c|c}
   \left( (-1)^{m-j} 
(x_i \prod_{q=1}^{j-1} 
   (x_i+y_q)-
  \x_i \prod_{q=1}^{j-1} 
   (\x_i+y_q))
\right)_{\substack{1 \leq i\leq n \\ 1 \leq j \leq m}} 
&  
Z_{\lambda}
\\[1.75em]
\hline\\
  \left(h_{\lambda'_i-n-i+j}(y_j,\dots,y_m)\right)_{\substack{1 \leq i\leq m-n+k-1 \\ 1\leq j\leq m}} 
  & {\normalfont\text{\huge0}}
\end{array} \right),
\end{equation*}
where
\[
  D = \prod_{i=1}^{n}
    (x_i-\x_i)
    \prod_{1 \leq i<j\leq n} 
(x_i+\x_i-x_j-\x_j).
  \]
Applying the elementary column transformations $C_j \rightarrow C_j + y_{j-1} C_{j-1}$, for $j \in [2,m]$ and 
using 
$h_n(y_{j-1},\dots,y_m)= h_n(y_{j},\dots,y_m)+y_{j-1} h_{n-1}(y_{j-1},\dots,y_m)$, 
$A$ is given by 
\[
\frac{(-1)^{a}}{D} \det \left( \begin{array}{c|c|c}
\left((-1)^{m-1} (x_i-\x_i)\right)_{1 \leq i \leq n}
&
   \left( N_{i,j}^2
\right)_{\substack{1 \leq i\leq n \\ 2 \leq j \leq m}} 
&
 Z_{\lambda}
\\&&\\\hline&&\\ 
 \left(h_{\lambda'_i-n-i+1}(Y)\right)_{1 \leq i\leq m-n+k-1 }
 &
 \left(h_{\lambda'_i-n-i+j}(y_{j-1},\dots,y_m)\right)_{\substack{1 \leq i\leq m-n+k-1 \\ 2 \leq j\leq m}} 
 &
 {\normalfont\text{\huge0}}
\end{array} \right),
\]
where 
\[
N_{i,j}^{s} = (-1)^{m-j} 
\left(x_i^s \prod_{q=1}^{j-s} 
   (x_i+y_q)-\x_i^s
   \prod_{q=1}^{j-s} 
   (\x_i+y_q)\right).
\]
Now again applying the elementary column transformations $C_j \rightarrow C_j + y_{j-2} C_{j-1}$, for $j \in [3,m]$, we see that 
\[
A=\frac{(-1)^{a}}{D} \det \left( \begin{array}{c|c|c}
\left((-1)^{m-j} (x_i^{j}-\x_i^j)
\right)_{\substack{1 \leq i\leq n \\ 1 \leq j \leq 2}} 
& \left( N_{i,j}^3 
\right)_{\substack{1 \leq i\leq n \\ 3 \leq j \leq m}}  & Z_{\lambda} 
\\&&\\\hline&&\\ 
 \left(h_{\lambda'_i-n-i+j}(Y)\right)_{\substack{1 \leq i\leq m-n+k-1\\1 \leq j \leq 2}}
 &
\left(h_{\lambda'_i-n-i+j}(y_{j-2},\dots,y_m)\right)_{\substack{1 \leq i\leq m-n+k-1 \\ 2 \leq j\leq m}} 
  & {\normalfont\text{\huge0}}
\end{array} \right).
\]
Proceeding in a similar fashion, we have 
\begin{equation}
\label{final}
A =
\frac{(-1)^{a}}{D} \det \left( \begin{array}{c|c}
\left( (-1)^{m-j} (x_i^{j}-\x_i^j) 
\right)_{\substack{1 \leq i\leq n \\ 1 \leq j \leq m}} 
&
Z_{\lambda}
\\[1.75em]
\hline\\
 \left(h_{\lambda'_i-n-i+j}(Y)
\right)_{\substack{1 \leq i\leq m-n+k-1 \\ 1 \leq j\leq m}} 
&
{\normalfont\text{\huge0}}
\end{array} \right).
\end{equation}
Observe that the expression in \eqref{final} is the same as the expression in \eqref{det-ortho-h}, which is equal to $\spo_{\lambda}(X;Y)$. 
This completes the proof.
\end{proof} 

\section{A generalization of Brent--Krattenthaler--Warnaar's identity}
\label{sec:bkw}
In this section, we prove a generalization of the following odd symplectic character identity due to Brent, Krattenthaler and Warnaar,
which was found in the study of discrete Mehta-type integrals~\cite{brent2016discrete}.
Recently, Okada gave {a} linear-algebraic proof~\cite{Okada2019ABF}.

\begin{thm}[{\cite[Theorem 1.2]{Okada2019ABF}}]
    Let $m$ and $n$ be positive integers with $n \leq m$ and $r$ a non-negative integer. 
    Then 
    \[
    \sum_{\lambda} z^{-r}  
    \sp^{(n,1)}_{\lambda}(X|z) 
    \, \sp^{(m,1)}_{(r^{m-n}) \cup \lambda}(Y|z)  =  
    \sp_{(\underbrace{r,\dots,r}_{m+n+1})}(X,Y,z),
    \]
where $\lambda$ runs over partitions of length at most $n+1$ such that $\lambda_1 \leq r$, and $(r^{m-n}) \cup \lambda$ {denotes} the partition $(\underbrace{r,\dots,r}_{m-n},\lambda_1,\dots,\lambda_{n+1})$. 
\end{thm}

We {recall} some lemmas before presenting the generalization (see \cref{thm:ole}). 
{We first recall the Cauchy-Binet formula.}
Given an $m \times n$ matrix 
$X=(x_{i,j})_{{1 \leq i \leq m,
1 \leq j \leq n}}$ and a subset $B \subset [n]=\{1,\dots, n\}$ 
of column indices, 
{let $X[B]$ be the submatrix of $X$ consisting of
the columns indexed by elements of $B$.}
\begin{lem}
\label{lem:CB}
    For two $m \times n$ matrices $X$ and $Y$, we have 
    \[
   \sum_{B} \det X[B] \det Y[B] = \det(XY^t), 
    \]
    where the sum runs over $m$-elements subset of $[n]$.
\end{lem}

\begin{lem}[{\cite[Lemma 3.1(b)]{Okada2019ABF}}]
\label{lem:ind-sp}
      Let $\lambda$ be a partition of length at most $n$ and $\lambda_1 \leq r$.
    Then 
    \[
    (x_1\cdots x_n)^r \sp_{\lambda}(X) 
    \]
    is a polynomial in $x_1,\dots,x_n$, and we have 
    \begin{equation*}
        \begin{split}
          [(x_1\cdots x_n)^r \sp_{\lambda}(X)]&|_{x_1=0}\\
        = &\begin{cases}
            (x_2 \cdots x_n)^r \sp_{(\lambda_2,\dots,\lambda_n)}(x_2,\dots,x_n) & \text{ if } \lambda_1=r,\\
            0 & \text{ otherwise,}
        \end{cases} 
        \end{split}
    \end{equation*}
{where $F|_{x_1=0}$ is the polynomial obtained by substituting
$x_1=0$ in $F$}.
\end{lem}
We need the following generalization of \cref{lem:ind-sp} in the proof of \cref{thm:ole}.
\begin{lem}
\label{lem:ind}
    Let $\lambda$ be a partition of length at most $n$ and $\lambda_1 \leq r$.
    Then 
    \[
    (x_1\cdots x_n)^r \spo_{\lambda}(X;z) 
    \]
    is a polynomial in $x_1,\dots,x_n$, and we have 
    \begin{equation*}
        \begin{split}
          [(x_1\cdots x_n)^r \spo_{\lambda}(X;z)]&|_{x_1=0}\\
        = &\begin{cases}
            (x_2 \cdots x_n)^r \spo_{(\lambda_2,\dots,\lambda_n)}(x_2,\dots,x_n;z) & \text{ if } \lambda_1=r,\\
            0 & \text{ otherwise,}
        \end{cases} 
        \end{split}
    \end{equation*}
    {where $F|_{x_1=0}$ is the polynomial obtained by substituting
$x_1=0$ in $F$}.
\end{lem}
\begin{proof}
By \cref{cor:m=1}, we have 
\[
\spo_{\lambda}(X;z) =
\frac{1}{D}
\det \left(
  x_i^{\lambda_j+n-j+1}-\x_i^{\lambda_j+n-j+1}+
  z \left(x_i^{\lambda_j+n-j} - \x_i^{\lambda_j+n-j}\right) 
\right),
\]   
where $D= 
\prod_{i=1}^{n}
    (x_i-\x_i)
    \prod_{1 \leq i<j\leq n} 
(x_i+\x_i-x_j-\x_j)$.
Let $A_{\lambda}(X) \coloneqq (a_{i,j})_{1 \leq i,j \leq n},$ {where} 
\begin{equation*}
    \begin{split}
        a_{i,j}= x_i^{r+n} \left(x_i^{\lambda_j+n-j+1}-\x_i^{\lambda_j+n-j+1}+
  z \left(x_i^{\lambda_j+n-j} - \x_i^{\lambda_j+n-j}\right)\right)&\\
 = x_i^{\lambda_j+r+2n-j+1}-x_i^{r-\lambda_j+j-1}+
  &z \left(x_i^{\lambda_j+r+2n-j} - x_i^{r-\lambda_j+j}\right).
    \end{split}
\end{equation*}
Then
\[
(x_1\cdots x_n)^r \spo_{\lambda}(X;z) = 
\frac{\det A_{\lambda}(X)}
{ x_1^n \cdots x_n^n
} \times \frac{1}{D}.
\]
Since 
{\[
\lambda_j+r+2n-j+1 > \lambda_j+r+2n-j > r-\lambda_j+j >
r-\lambda_j+j-1 \geq 0,
\]}
$(x_1 \cdots x_n)^r \spo_{\lambda}(X;z)$ is a polynomial in $x_1, \dots, x_n$. 
Also, note that $r-\lambda_j+j-1=0$ if and only if $j=1$ and $\lambda_1=r$. Therefore, 
\[
a_{1,j}|_{x_1=0}=\begin{cases}
    -1 & \lambda_1=r \text{ and } j=1,\\
    0 & \text{ otherwise,}
\end{cases}
\]
{and}
\[
(\det A_{\lambda}(X))|_{x_1=0} = 
\begin{cases}
-\det \left(a_{i+1,j+1}\right)_{1 \leq i,j \leq n-1} & \lambda_1=r, \\
0 & \text{ otherwise.}
\end{cases}
\]
Observe that 
\[
\det \left(a_{i+1,j+1}\right)_{1 \leq i,j \leq n-1}
=
(x_2 \cdots x_n) \det A_{(\lambda_2,\dots,\lambda_n)}(x_2,\dots,x_n).
\]
If $\lambda_1=r$, then 
we have 
\begin{equation*}
        \begin{split}
          [(x_1\cdots x_n)^r \spo_{\lambda}(X;z)]|_{x_1=0}
        &  = \frac{-(x_2 \cdots x_n) \det A_{(\lambda_2,\dots,\lambda_n)}(x_2,\dots,x_n)}{-(x_2 \cdots x_n)^n \prod_{i=2}^n (x_i-\x_i)
        \prod_{2 \leq i<j\leq n} 
(x_i+\x_i-x_j-\x_j)} 
          \\
      &  =  (x_2 \cdots x_n)^r \spo_{(\lambda_2,\dots,\lambda_n)}(x_2,\dots,x_n;z). 
         \end{split}
    \end{equation*}
This completes the proof.
\end{proof}
\begin{lem}[{\cite[Lemma 3.2]{Okada2019ABF}}]
\label{lem:P}
 Let $p(x,y,z,a,b)$ be {the} rational function in $x$, $y$, $z$, $a$ and $b$ given by
 \begin{equation*}
\label{p-1}
    \begin{split}
        p(x,y,&z,a,b)=\\
    & \frac{(1-xz)(1-yz)}{1-xy} -
    a \frac{(x-z)(1-yz)}{x-y} +
    b \frac{(1-xz)(y-z)}{x-y} -
    ab \frac{(x-z)(y-z)}{1-xy}.
    \end{split}
\end{equation*}
Let $\mathbf{x}=(x_1,\dots,x_n),$  $\mathbf{y}=(y_1,\dots,y_n),$ $\mathbf{a}=(a_1,\dots,a_n),$ $\mathbf{b}=(b_1,\dots,b_n),$ $z$ and $c$ be determinates. We define an $(n+1) \times (n+1)$ matrix $C=C(\mathbf{x},\mathbf{y},z;\mathbf{a},\mathbf{b},c)=(C_{i,j})_{1 \leq i,j \leq n+1}$ by putting
\[
C_{i,j}= \begin{cases}
    p(x_i,y_j,z,a_i,b_j) & 1 \leq i,j \leq n, \\
    1-a_i & 1 \leq i \leq n, j=n+1, \\
    1-b_j & i=n+1, 1 \leq j \leq n, \\
  \ds  \frac{1-c}{1-z^2} & i=j=n+1,
\end{cases}
\]
and a $(2n+1) \times (2n+1)$ matrix $V=V(\mathbf{x},\mathbf{y},z;\mathbf{a},\mathbf{b},c)=(V_{i,j})_{1 \leq i,j \leq 2n+1}$ by putting 
\[
V_{i,j}=\begin{cases}
    x_i^{j-1} - a_i x_i^{2n+1-j} & 1 \leq i \leq n,\\
    y_{i-n}^{j-1} - b_{i-n} y_{i-n}^{2n+1-j} & n+1 \leq i \leq 2n,\\
    z^{j-1}-cz^{2n+1-j} & i=2n+1,
\end{cases}
\]
for $1 \leq j \leq 2n+1$.
Then 
\begin{equation}
\label{C}
    \det C = 
\frac{(-1)^n}{\ds (1-z^2) \prod_{i=1}^n \prod_{j=1}^n (x_i-y_j)(1-x_iy_j)} \det V.
\end{equation}
 
\end{lem}
\begin{lem} 
\label{lem:q}
Let $q(x,y,z,a,b)$ be {the} rational function in $x$, $y$, $z$, $a$ and $b$ given by
\begin{equation}
\label{p}
    \begin{split}
        q(x,y,&z,a,b)=\\
    & \frac{(1+xz)(1+yz)}{1-xy} -
   a \frac{(x+z)(1+yz)}{x-y} +
    b \frac{(1+xz)(y+z)}{x-y} -
    ab \frac{(x+z)(y+z)}{1-xy}.
    \end{split}
\end{equation}
Then 
\[
\det \left( q(x_i,y_j,z,a_i,b_j) \right)_{1 \leq i,j \leq n} = 
\frac{(-1)^n}{\ds \prod_{i=1}^n \prod_{j=1}^n (x_i-y_j)(1-x_iy_j)} \det W,
\]
where $W=(W_{i,j})_{1 \leq i,j \leq 2n+1}$ such that 
\begin{equation}
\label{W}
W_{i,j}=\begin{cases}
    x_i^{j-1} - a_i x_i^{2n+1-j} & 1 \leq i \leq n\\
    y_{i-n}^{j-1} - b_{i-n} y_{i-n}^{2n+1-j} & n+1 \leq i \leq 2n\\
    (-z)^{2n+1-j} & i=2n+1,
\end{cases}    
\end{equation}
for $1 \leq j \leq 2n+1$.
\end{lem}
\begin{proof} Recall the notions of $p(x,y,z,a,b)$, $(n+1) \times (n+1)$ matrix $C$ and $(2n+1) \times (2n+1)$ matrix $V$ from \cref{lem:P}. 
    Observe that 
    \[
    q(x,y,z,a,b) = p(x,y,-z,a,b)
    \]
    and
    \[
    \det \left( q(x_i,y_j,-z,a_i,b_j) \right)_{1 \leq i,j \leq n} = \frac{(1-z^2)}{-c} \left( \det C - \det C|_{c=0} \right).
    \]
    Then by \eqref{C}, we have 
    \[
    \det \left( q(x_i,y_j,-z,a_i,b_j) \right)_{1 \leq i,j \leq n} =  \frac{(-1)^n}{\ds -c \prod_{i=1}^n \prod_{j=1}^n (x_i-y_j)(1-x_iy_j)} \left( \det V - \det V|_{c=0} \right)
    \]
    Therefore,
\[
    \det \left( q(x_i,y_j,z,a_i,b_j) \right)_{1 \leq i,j \leq n} =  \frac{(-1)^n}{\ds \prod_{i=1}^n \prod_{j=1}^n (x_i-y_j)(1-x_iy_j)} \det W,
    \]
    where $W$ is given in \eqref{W}. This completes the proof.
\end{proof}
\begin{thm}
\label{thm:ole}
Let $m$ and $n$ be positive integers with
$n \leq m$, and $r$ a nonnegative integer. Then we have
\begin{equation}
\label{BKW}
    \begin{split}
     \sum_{\lambda} z^r \spo_{\lambda}(X;z) \spo_{(r^{m-n}) \cup \lambda}&(Y;z) \\
     =  
    \sum_{j=1}^{m+n+1} & z^{r+m+n+1-j} \sp_{\big(\underbrace{r,\dots,r}_{j-1},
    \underbrace{r-1,\dots,r-1}_{m+n+1-j}\big)}(X,Y),  
    \end{split}
\end{equation}
where $\lambda$ runs over partitions of length at most $n$ such that $\lambda_1 \leq r$ and $(r^{m-n}) \cup \lambda$ 
{denotes} the partition 
$(\underbrace{r,\dots,r}_{m-n},\lambda_1,\dots,\lambda_{n})$. 
\end{thm}
\begin{proof} Assume $m=n$. Let $M=r+n$ and consider $n \times M$ matrices $X=(X_{i,p})_{\substack{1 \leq i \leq n\\
1 \leq p \leq M}}$ and $Y=(Y_{i,p})_{\substack{1 \leq i \leq n\\
1 \leq p \leq M}}$
such that
\[
\begin{cases}
  X_{i,p} &= \quad x_i^{p}-\x_i^{p} + z(x_i^{p-1}-\x_i^{p-1}), \\ 
  Y_{i,p} &= \quad y_i^{p}-\bar{y}_i^{p} + z(y_i^{p-1}-\bar{y}_i^{p-1}),
\end{cases}
\]
for $1 \leq i \leq n$, $1 \leq p \leq M$.
{By using the formula for the geometric sum, it is immediate that}
the {$(i,j)$-th} entry of $XY^t$ is given by 
\[
\x_i^{M} \bar{y}_j^{M} q(x_i,y_j,z,x_i^{2M},y_j^{2M}),
\]
where 
\begin{multline*}
    q(x,y,z,a,b)=\\
    \frac{(1+xz)(1+yz)}{1-xy} -
    a \frac{(x+z)(1+yz)}{x-y} +
    b \frac{(1+xz)(y+z)}{x-y} -
    ab \frac{(x+z)(y+z)}{1-xy}.
\end{multline*}
Therefore,
\begin{equation}
\label{xy}
    \det XY^t = (\x_1 \cdots \x_n)^{M} 
 (\bar{y}_1 \cdots \bar{y}_n)^{M}  
 \det \left( q(x_i,y_j,z,x_i^{2M},y_j^{2M}) \right).
\end{equation}
By \cref{lem:q}, we have
\begin{equation}
    \label{detq}
    \det \left( q(x_i,y_j,z,x_i^{2M},y_j^{2M}) \right) =
 \frac{(-1)^n}{\ds \prod_{i=1}^n \prod_{j=1}^n (x_i-y_j)(1-x_iy_j)} \det W^{(M)},
\end{equation}
where 
\[
\left(W^{(M)}\right)_{i,j} = \begin{cases}
    x_i^{j-1} - x_i^{2M+2n+1-j} & 1 \leq i \leq n\\
    y_{i-n}^{j-1} - y_{i-n}^{2M+2n+1-j} & n+1 \leq i \leq 2n\\
    (-z)^{2n+1-j} & i=2n+1,
\end{cases}
\]
for $1 \leq j \leq 2n+1$.
Suppose $I_n(\lambda)={\{}\lambda_1+n,\dots,\lambda_n+1{\}}$. Then by \cref{cor:m=1} and the Cauchy--Binet formula in \cref{lem:CB}, the left hand side of \eqref{BKW} is given by
\begin{equation}
\label{big}
    \begin{split}
     \sum_{\lambda} z^r \spo_{\lambda}(X;z) \spo_{ \lambda}(Y;z)  =&
   \ds \sum_{\lambda} \frac{z^r}{D_x D_y} 
    \det X[I_n(\lambda)]
    \, {\det Y[I_n(\lambda)]}  
    \\
    =   \frac{z^r }{D_x D_y}
     \det(XY^t) = &  \frac{(-1)^n z^r (\x_1 \cdots \x_n)^{M} 
 (\bar{y}_1 \cdots \bar{y}_n)^{M} }{D_x
D_y \ds \prod_{i=1}^n \prod_{j=1}^n (x_i-y_j)(1-x_iy_j) } \times  \det W^{(M)},
    \end{split}
\end{equation}
    where the last equality uses \eqref{xy} and \eqref{detq} and $D_x=\prod_{i=1}^{n}
    (x_i-\x_i)
    \prod_{1 \leq i<j\leq n} 
(x_i+\x_i-x_j-\x_j)$, $D_y=\prod_{i=1}^{n}
    (y_i-\bar{y}_i)
    \prod_{1 \leq i<j\leq n} 
(y_i+\bar{y}_i-y_j-\bar{y}_j)$. Simplifying \eqref{big}, we have 
    \[
 \sum_{\lambda} z^r \spo_{\lambda}(X;z) \spo_{ \lambda}(Y;z)  = \frac{(-1)^{r+n} \prod_{i=1}^n x_i^{n} y_i^{n}}{D_x D_y \ds \prod_{i=1}^n \prod_{j=1}^n (x_i-y_j)(1-x_iy_j)} 
    \det W_0,
    \]
    where 
    \[
    (W_0)_{i,j}=\begin{cases}
    x_i^{M+n+1-j} - \x_i^{M+n+1-j} & 1 \leq i \leq n,\\
   y_i^{M+n+1-j} - \bar{y}_i^{M+n+1-j} & n+1 \leq i \leq 2n,\\
    (-z)^{M+n+1-j} & i=2n+1,
\end{cases}
    \]
for $1 \leq j \leq 2n+1$. Note that
\[
\frac{\ds \prod_{i=1}^n \prod_{j=1}^n (x_i-y_j)(1-x_iy_j)}{\prod_{i=1}^n x_i^{n} y_i^{n}} = 
\prod_{i=1}^n \prod_{j=1}^n  (y_j+\bar{y}_j-x_i-\x_i)
= (-1)^{n^2} \prod_{i=1}^n \prod_{j=1}^n  (x_i+\x_i-y_j-\bar{y}_j).
\]
Therefore, 
\begin{equation*}
    \begin{split}
        \sum_{\lambda} z^r \spo_{\lambda}(X;z) \spo_{ \lambda}(Y;z) &= 
        \frac{(-1)^{r+n^2+n}}{D_x D_y \prod_{i=1}^n \prod_{j=1}^n  (x_i+\x_i-y_j-\bar{y}_j)} 
        \det W_0 \\
       & =
        \sum_{j=1}^{2n+1} 
        z^{r+2n+1-j}
        \sp_{\big(\underbrace{r,\dots,r}_{j-1},\underbrace{r-1,\dots,r-1}_{2n+1-j}\big)}(X,Y),
    \end{split}
\end{equation*}
where the last equality uses \eqref{spdef}. This completes the proof in the case $m=n$. Now we prove the general case by downward induction on   
$n$. 
Multiplying both sides of \eqref{BKW} by $(x_1\cdots x_n y_1\cdots y_m)^r$ and then substituting $x_1=0$, by \cref{lem:ind-sp} and \cref{lem:ind}, we have 
\begin{equation*}
    \begin{split}
     \sum_{\lambda} z^r (x_2\cdots x_n)^r \spo_{(\lambda_2,\dots,\lambda_n)}(x_2,\dots,x_n;z) (y_1\cdots y_m)^r \spo_{(r^{m-n}) \cup \lambda}(Y&;z) \\
     =  
    \sum_{j=2}^{m+n+1}  z^{r+m+n+1-j} (x_2 \cdots x_n y_1\cdots y_m)^r \sp_{\big(\underbrace{r,\dots,r}_{j-2},\underbrace{r-1,\dots,r-1}_{m+n+1-j}\big)} &(x_2,\dots,x_n,y_1,\dots,y_m),   
    \end{split}
\end{equation*}
where the summation on the left hand side is over partitions $\lambda$ of length at most $n$ such that $\lambda_1=r$. Since $(r^{m-n}) \cup \lambda=(r^{m-n+1}) \cup (\lambda_2,\dots,\lambda_n)$ for such a partition, we have 
\begin{equation*}
    \begin{split}
     \sum_{\mu} z^r \spo_{\mu}(x_2,\dots,x_n;z) & \spo_{(r^{m-n+1}) \, \cup \, \mu}(Y;z) \\
     =  
    \sum_{j=1}^{m+n} & z^{r+m+n-j}  \sp_{\big(\underbrace{r,\dots,r}_{j-1},\underbrace{r-1,\dots,r-1}_{m+n-j}\big)} (x_2,\dots,x_n,y_1,\dots,y_m),   
    \end{split}
\end{equation*}
where $\mu$ runs over all partitions of length at most $n-1$ such that $\mu_1 \leq r$. This completes the proof.
\end{proof}
\section*{Declarations}
The authors have no competing interests to declare that are relevant to the content of this article.
\section*{Acknowledgements}
We thank A. Ayyer and A. Stokke for very helpful discussions. We acknowledge support from the Department of Science and Technology grant CRG/2021/001592.
 \bibliographystyle{alpha}
  \bibliography{Bibliography}
\end{document}